\documentclass[10pt,a4paper]{article}
\usepackage[latin1]{inputenc}
\usepackage{amsmath}
\usepackage{amsfonts}
\usepackage{amssymb}
\usepackage{amsthm}
\usepackage{mathrsfs}
\usepackage{latexsym}
\usepackage{amsmath,amscd}
\usepackage[all]{xy}
\usepackage{mathtools}

\newtheorem{definition}{Definition}

\newtheorem{lemma}[definition]{Lemma}
\newtheorem{proposition}[definition]{Proposition}

\newtheorem*{main}{Theorem A}
\newtheorem*{main2}{Theorem B}
\newtheorem*{main3}{Corollary C}

\def\XXint#1#2#3{{\setbox0=\hbox{$#1{#2#3}{\int}$ }
\vcenter{\hbox{$#2#3$ }}\kern-.6\wd0}}

\begin{document}
\section*{Non-amenability and visual Gromov hyperbolic \\ spaces}

\begin{Large}
{Juhani Koivisto} \\ 
\end{Large} \\ \emph{Department of Mathematics and Computer Science, \\ University of Southern Denmark \\ Campusvej 55, DK-5230 Odense M, Denmark \\ e-mail: koivisto@imada.sdu.dk} \\

\textbf{Abstract.} We prove that a uniformly coarsely proper hyperbolic cone over a bounded metric space consisting of a finite union of uniformly coarsely connected components each containing at least two points is non-amenable and apply this to visual Gromov hyperbolic spaces. \\ \par 
\textbf{Mathematics subject classification (2000):} 53C23, 20F67 \\

\textbf{Key words:} Hyperbolic spaces, isoperimetry, amenability. \\
 
\section{Introduction}
 
A metric space $(X,d)$ is \emph{uniformly coarsely proper} if there exist $N \colon (0,\infty) \times (0,\infty) \rightarrow \mathbb{N}$ and a constant $r_b > 0$ such that for all $R > r > r_b$ every open ball of radius $R$ in $X$ can be covered by $N(R,r)$ open balls of radius $r$ in $X$. A subset $\Gamma \subseteq X$ is \emph{($\mu$-)cobounded} if there exists a constant $\mu>0$ such that $d(x,\Gamma) < \mu$ for all $x \in X$ and \emph{uniformly locally finite} if there exists $N \colon (0,\infty) \rightarrow \mathbb{N}$ such that the cardinality $\# (\Gamma \cap B(x,r)) \leq N(r)$ for all $0 < r < \infty$ and all $x \in X$. As usual $B(x,r) = \lbrace y \in X \colon d(x,y)<r\rbrace$. A \emph{quasi-lattice} in $(X,d)$ is a cobounded uniformly locally finite subset $\Gamma \subseteq X$, and $(X,d)$ is uniformly coarsely proper if and only if it has a quasi lattice \cite[Proposition 3.D.16]{CorHar}. A uniformly coarsely proper space $(X,d)$ is now said to be \emph{non-amenable} if there exist a quasi-lattice $\Gamma \subseteq X$ and constants $C>0$ and $r>0$ such that for any finite subset $F \subseteq \Gamma$ 
\begin{equation}\# F \leq C \# \partial_r F  \nonumber 
\end{equation} where $\partial_r F = \lbrace x \in \Gamma \colon d(x,F) < r \, \mathrm{and} \, d(x, \Gamma \setminus F) < r\rbrace$. \par 
A complete geodesic Gromov hyperbolic Riemannian manifold (or metric graph) with bounded local geometry and quasi-pole is non-amenable if its Gromov boundary consists of finitely many connected components of strictly positive diameter; see \cite{Cao}. We show more generally that a uniformly coarsely proper hyperbolic cone over any bounded metric space with finitely many uniformly coarsely connected components each containing at least two points is non-amenable; and hence that any uniformly coarsely proper visual Gromov hyperbolic space is non-amenable if its Gromov boundary consists of finitely many uniformly coarsely connected components of strictly positive diameter. The terminology and results are in detail as follows.  \par  
A space $(X,d)$ is \emph{Gromov hyperbolic} if it satisfies for some $\delta \in [0, \infty)$ the Gromov product inequality $$(x \vert z)_w \geq \min \lbrace (x \vert y)_w, (y \vert z)_w \rbrace - \delta$$ for all $x,y,z,w \in X$. 
The \emph{hyperbolic cone} over a bounded metric space $(Z,d)$ containing at least two points is the metric space $(\mathcal{H}(Z),\rho)$ where $\mathcal{H}(Z)=Z \times [0,\infty)$, $$\rho((x,t),(y,s)) = 2 \log \left( \dfrac{d(x,y) + \max \lbrace e^{-t},e^{-s} \rbrace D}{e^{-(s+t)/2} D} \right),$$
and $D = \mathrm{diam}(Z)$. 
A space $(X,d)$ is \emph{$\varepsilon$-coarsely connected} for $\varepsilon > 0$ if for every $x,y \in X$ there exists an \emph{$\varepsilon$-sequence} from $x$ to $y$ in $X$, by which we mean a finite sequence of points $x=x_0, \dots, x_n=y$ in $X$ such that $d(x_{i},x_{i+1}) \leq \varepsilon$ for all $0 \leq i \leq n-1$. If $(X,d)$ is $\varepsilon$-coarsely connected for all $\varepsilon >0$ we say that $(X,d)$ is \emph{uniformly coarsely connected}; a \emph{uniformly coarsely connected component} of $(X,d)$ is any subset of the form $C(x,X) = \bigcup \lbrace A \colon x \in A \subseteq X,\,A \textrm{ uniformly coarsely connected}\rbrace$. If $(X,d)$ is compact its uniformly coarsely connected components are its connected components. \par Our main result is the following coarse generalisation of \cite[Theorem 3.2]{Cao}.
\begin{main}
Let $(\mathcal{H}(Z),\rho)$ be the hyperbolic cone over a bounded space $(Z,d)$. If $(\mathcal{H}(Z), \rho)$ is uniformly coarsely proper and $(Z,d)$ consists of a finite union of uniformly coarsely connected components each containing at least two points then $(\mathcal{H}(Z), \rho)$ is non-amenable. 
\end{main}
A space is visual if there exists a basepoint so that every point in the space is contained in the image of some roughly geodesic ray issuing from it; see Section \ref{metricspaces}. This gives the following generalisation of \cite[Main Theorem 1.1]{Cao}.
\begin{main2}
If $(X,d)$ is a uniformly coarsely proper visual Gromov hyperbolic space whose Gromov boundary consists of a finite union of uniformly coarsely connected components each containing at least two points then $(X,d)$ is non-amenable.
\end{main2}
\begin{proof}
As $X$ is visual Gromov hyperbolic its boundary $\partial X$ is a bounded metric space and there exists a rough-similarity $f \colon X \rightarrow \mathcal{H}(\partial X)$; see \cite[Proposition 6.2, Theorem 8.2]{BS}. 
 
Since $\partial X$ consists of finitely many uniformly coarsely connected components each containing at least two points $\mathcal{H}(\partial X)$ is non-amenable by Theorem A since uniformly coarsely proper is a quasi-isometry invariant by \cite[Corollary 3.D.17]{CorHar}. The claim now follows as non-amenability is a quasi-isometry invariant by \cite[Corollary 2.2]{ABW}.
\end{proof}

The Gromov boundary of a locally compact compactly generated hyperbolic group is compact so all of its uniformly coarsely connected components are connected; and if it consists of finitely many connected components containing at least two points, it consists of exactly one connected component containing these points; see for example \cite[Section 2.C]{Cor2}. 
\begin{main3}
Let $G$ be a locally compact compactly generated hyperbolic group whose boundary is connected and contains at least two points. Then $G$ is not geometrically amenable.
\end{main3}
\begin{proof}Suppose $G$ is compactly generated by $S \subseteq G$ and write $(G,d_S)$ for the corresponding word metric space noting that it is uniformly coarsely proper; see Lemma \ref{where}. By the characterisation of hyperbolic groups \cite[Corollary 2.6]{Mon} and the \v{S}varc-Milnor Lemma \cite[Theorem 4.C.5]{CorHar} there exists a quasi-isometry $f: (G,d_S) \rightarrow (X,d)$ where $(X,d)$ is some proper geodesic Gromov hyperbolic space. This induces a power-quasisymmetry $\partial f \colon \partial G \rightarrow \partial X$; see \cite[Theorem 6.5]{BS}. Since $\partial f$ is a homeomorphism $\partial X$ is connected and contains at least two points and $(X,d)$ is non-amenable by Theorem B. In particular $(G,d_S)$ is non-amenable. The claim now follows from \cite[Corollary 11.14]{RT}. \end{proof} 

\subsection{Organisation of the paper}
In section \ref{metricspaces}, we recall the terminology used for metric spaces not covered in the introduction and prove some folklore results claiming no originality whatsoever. Section \ref{HC} contains the gist of the paper: here we cover the hyperbolic cone construction; Cao's graph approximation; and prove Theorem A adapting techniques from Cao \cite{Cao} and V\"ah\"akangas \cite{AV}. \\   

\textbf{Acknowledgements} I would like to thank Ilkka Holopainen for his advice and for providing me with unpublished notes written by Aleksi V\"ah\"akangas in 2007 on global Sobolev inequalities on Gromov hyperbolic spaces; Jussi V\"ais\"al\"a for providing me with the letters of correspondence between him and Oded Schramm from the end of 2004 with regard to the paper Embeddings of Gromov hyperbolic spaces, and to whom Lemma \ref{vais} is attributed; Piotr Nowak for many enjoyable discussions on growth homology; Yves de Cornulier for bringing his quasi-survey \cite{Cor2} to my attention; and Pekka Pankka for several suggestions on how to improve the text. I would also like to thank the Technion for its hospitality during my stay from January to May 2014, Uri Bader and Tobias Hartnick for many stimulating conversations, and Eline Zehavi for all her help during this stay. Last, I would like to thank the Academy of Finland, projects 252293 and 271983, and ERC grant 306706, for financial support.

\section{Basic notions and folklore} \label{metricspaces}
 
A subset $N \subseteq X$ in $(X,d)$ is \emph{($\mu$-)separated} if there exists a constant $\mu > 0$ such that $d(x,y) \geq \mu$ whenever $x,y \in N$ are distinct. A \emph{maximal $\mu$-net} in $(X,d)$ is a $\mu$-separated $\mu$-cobounded subset $N \subseteq X$. Note that a maximal $\mu$-net $N \subseteq X \neq \emptyset$ always exists for any $\mu > 0$ by Zorn's lemma. \par
A function $f \colon X \rightarrow X'$ between $(X,d)$ and $(X,d')$ is a \emph{$(\lambda,\mu)$-quasi-isometric embedding} if there exist constants $\lambda \geq 1$ and $\mu \geq 0$ such that $$ \lambda^{-1} d(x,y) - \mu \leq d'(f(x),f(y)) \leq \lambda d(x,y) + \mu$$ for all $x,y \in X$, and \emph{$\mu$-essentially surjective} if $d(x',f(X)) \leq \mu$ for all $x' \in X'$. A $\mu$-essentailly surjective $(\lambda,\mu)$-quasi-isometric embedding $f \colon X \rightarrow X'$ is a \emph{$(\lambda, \mu)$-quasi-isometry} and $(X,d)$ and $(X',d')$ are said to be \emph{quasi-isometric}. 
A $(\lambda, \mu)$-quasi-isometry $f \colon X \rightarrow X'$ is \emph{a $(\lambda, \mu)$-rough similarity} if $$ \lambda d(x,y) - \mu \leq d'(f(x),f(y)) \leq \lambda d(x,y) + \mu$$ for all $x,y \in X$. \par 
Abbreviating "from $x$ to $y$" by $x \curvearrowright y$, we say that a $(1,\mu)$-quasi-isometric embedding $\gamma \colon [a,b] \rightarrow X$ from a compact interval $[a,b] \subseteq \mathbb{R}$ is a \emph{$\mu$-rough geodesic} $x \curvearrowright y$ where $x=\gamma(a)$ and $y= \gamma(b)$. A $(1,\mu)$-quasi-isometric embedding $\gamma \colon [0,\infty) \rightarrow X$ is called a \emph{$\mu$-roughly geodesic ray} issuing from $\gamma(0)$. 
A $\mu$-rough geodesic $\gamma \colon x \curvearrowright y$ can always be parametrised by $d(x,y)$. 
\begin{lemma} \label{vais}
Given a $\mu$-rough geodesic $\gamma \colon [a,b] \rightarrow X$ $x \curvearrowright y$ there exists a $2 \mu$-rough geodesic $\beta \colon [0,d(x,y)] \rightarrow X$ $x \curvearrowright y$.
\end{lemma}
\begin{proof}
Write $R=d(x,y)$ and assume without loss of generality that $[a,b]=[0,b]$. First assume $b<R$. Extend $\gamma$ to $\beta \colon [0,R] \rightarrow X$ by $\beta(t)=\gamma(t)$ for $0 \leq t \leq b$ and $\beta(t)=\gamma(b)=y$ for $b \leq t \leq R$. Restricted to $0 \leq t \leq b$ the function $\beta$ is trivially a $2 \mu$-rough geodesic $x \curvearrowright y$. Next, consider the case when $0 \leq s \leq b < t \leq R$. Now, $$d(\beta(s), \beta(t)) = d(\gamma(s), \gamma(b)) \leq (b-s) + \mu \leq (t-s) + \mu.$$ On the other hand, since $\gamma$ is a $\mu$-rough geodesic $ \vert R-b \vert \leq \mu$, in particular since $t \leq R$ it follows from $R-b \leq \mu$ that $t-b \leq \mu$. As $t-b > 0$, $\vert t-b\vert \leq \mu$, and so also $$d(\beta(s), \beta(t)) \geq \vert s-b \vert - \mu = \vert s-t \vert - \vert t-b \vert - \mu \geq \vert s-t\vert - 2 \mu.$$ Finally, if $b \leq s,t \leq R$, again since $R-b \leq \mu$ it follows that $0 \leq s-b \leq \mu$ and $0 \leq t-b \leq \mu$. In particular, $\vert t-s\vert \leq \vert t-b\vert+ \vert b-s\vert \leq 2 \mu$ and we conclude that $\beta$ is a $2 \mu$-rough geodesic. \par 
Next, assume $R < b$. This time define $\beta \colon [0,R] \rightarrow X$ by $\beta(t)= \gamma(t)$ for $0 \leq t < R$, and $\beta(R)= \gamma(b)=y$. We claim that $\beta$ is a $2 \mu$-rough geodesic $x \curvearrowright y$. Clearly $\beta \colon x \curvearrowright y$, and since $\vert R-b\vert \leq \mu$ whenever $t < R$, $$d(\beta(t), \beta(R)) \leq \vert t-b \vert + \mu \leq \vert t-R\vert + \vert R-b\vert+ \mu \leq \vert t-R\vert + 2\mu,$$ and similarly, $$d(\beta(t), \beta(R)) \geq \vert t-b\vert - \mu \geq \vert t-R\vert - \vert R-b\vert- \mu \geq \vert t-R\vert - 2 \mu, $$ so $\beta$ is a $2 \mu$-rough geodesic as claimed.
\end{proof}

A space $(X,d)$ is \emph{($\mu$-)roughly geodesic} if for every $x,y \in X$ there exists a $\mu$-rough geodesic $\gamma \colon [0,d(x,y)] \rightarrow X$ $x \curvearrowright y$, and \emph{($\mu$-)visual} if there exists $o \in X$ such that every point in $X$ is contained in the image of a $\mu$-roughly geodesic ray issuing from $o$. \par

We end this section with a few clarifying remarks. A space $(X,d)$ has \emph{bounded growth at some scale} if there exist constants $R>r>0$ and $N \in \mathbb{N}$ such that any open ball of radius $R$ in $X$ can be covered by $N$ open balls of radius $r$ in $X$; see \cite{BS}. This is used by Cao in the context of geodesic spaces in \cite{Cao} and we note the following.

\begin{lemma} \label{improvement}
If $(X,d)$ is a length space then it is uniformly coarsely proper if and only if it has bounded growth at some scale.
\end{lemma}
\begin{proof}
If $(X,d)$ is uniformly coarsely proper it has bounded growth at some scale. So suppose $(X,d)$ has bounded growth at some scale $R>r>0$ and cover $B(x,R)$ by $N$ open balls $B(x_1,r), \dots, B(x_i,r), \dots, B(x_N,r)$. Since $(X,d)$ is a length space, for each $y \in B(x,2R-r)$ there exists $y' \in B(x,R)$ such that $d(y,y') \leq R-r$. Thus, for any $y \in B(x,2R-r)$ we can find $y' \in B(x,R)$ and $x_i$ as above such that $$d(x_i, y) \leq d(x_i,y') + d(y',y) \leq r + R - r = R.$$ In other words, $B(x_1,R), \dots, B(x_N,R)$ cover $B(x,2R-r)$, and it follows that $B(x,2R-r)$ can be covered by $N^2$ balls of radius $r$. By induction, for any $n \in \mathbb{N}$, the ball $B(x,(n+1)R-nr)$ can be covered by $N^{n+1}$ open balls of radius $r$. 
\end{proof}

Being uniformly coarsely proper is an invariant under metric coarse equivalence by \cite[Corollary 3.D.17]{CorHar}. For the readers convenience, we give a short proof for quasi-isometries proving an explicit estimate for the scale as well.

\begin{lemma} \label{growthlargescale}
Suppose $f \colon X \rightarrow X'$ is a $(\lambda, \mu)$-quasi-isometry between $(X,d)$ and $(X',d')$. If $(X,d)$ is uniformly coarsely proper for $R>r>r_{b}$ then $(X',d')$ is uniformly coarsely proper for $R' >r'> \lambda \mu + \mu + r_{b} \lambda$. 
\end{lemma}
\begin{proof}
Since $f \colon X \rightarrow X'$ is a $(\lambda,\mu)$-quasi-isometry, it is has a $(\lambda, 3 \lambda \mu)$-quasi-isometric coarse inverse $g \colon X' \rightarrow X$ where $d'(f(g(y)),y) \leq \lambda \mu$ for all $y \in X'$; see \cite{V}. Let $R' > \lambda \mu + \mu + r_{b} \lambda$. We claim that any $B(y,R') \subseteq X'$ can be covered by $N'(R',r')$ open balls or radius $R'>r'> \lambda \mu + \mu + r_{b}\lambda$. To begin $$g(B(y,R')) \subseteq B(g(y), \lambda R' + 3\lambda \mu),$$ and the latter can be covered by $N=N(\lambda R' + 3 \lambda \mu, s)$ balls $B(x_1,s), \dots, B(x_N,s)$ of radius $s > r_{b}$ as $X$ is uniformly coarsely proper. Choose $s = \lambda^{-1} r' - \lambda^{-1} \mu - \mu $. Now $E= f(B(g(y), \lambda R' + 3 \lambda \mu))$ is covered by the sets $f(B(x_i,s))$ and as $$f(B(x_i,s)) \subseteq B(f(x_i), \lambda s + \mu) = B(f(x_i), r'- \lambda \mu)$$ the balls $B(f(x_i), r' - \lambda \mu)$ cover $E$. Now since $d'(f(g(y)),y) \leq \lambda \mu$ for all $y \in X'$ \begin{align} B(y,R') \subseteq \lbrace x \in X' \colon d'(x,f(g(B(y,R')))) \leq \lambda \mu \rbrace \subseteq \lbrace x \in X' \colon d'(x,E) \leq \lambda \mu \rbrace, \nonumber \end{align} and as $E$ is covered by the balls $B(f(x_i), r' - \lambda \mu)$, the set $\lbrace x \in X' \colon d'(x,E) \leq \lambda \mu \rbrace$ is covered by the balls $B(f(x_i),r')$ covering $B(y,R)$ as well. Letting $N'(R',r')=N$ it follows that $(X',d')$ is uniformly coarsely proper for $R' >r'> \lambda \mu + \mu + r_{b} \lambda$. 
\end{proof}
The following appears in the proof of Corollary C.
\begin{lemma} \label{where}
If $G$ is locally compact and compactly generated by $S$ then $(G,d_S)$ is uniformly coarsely proper.
\end{lemma}
\begin{proof}
By \cite[Proposition 6.6]{RT} the word metric space $(G,d_S)$ is quasi-isometric to a connected metric graph $(X,d)$ of bounded valency implying it has bounded growth at some scale. Since $(X,d)$ is geodesic this implies that $(X,d)$ is uniformly coarsely proper by Lemma \ref{improvement}. The claim now follows since being uniformly coarsely proper is a quasi-isometry invariant. 
\end{proof}

\section{The hyperbolic cone} \label{HC}
The original construction of the hyperbolic cone is due to Bonk and Schramm who introduced in \cite{BS} the metric space $(\mathrm{Con}(Z),\rho_{BS})$ over a bounded metric space $(Z,d)$ where $\mathrm{Con}(Z) = Z \times (0,D]$ for $D= \mathrm{diam}(Z)$ assuming that $D>0$, and $${\rho}_{BS}((x,t),(y,s)) = 2 \log\left( \dfrac{d(x,y)+ \max \lbrace t,s \rbrace}{\sqrt{ts}}\right).$$ We note that $(\mathrm{Con}(Z),\rho_{BS})$ and $(\mathcal{H}(Z),\rho)$ are isometric where the isometry from $(\mathrm{Con}(Z),\rho_{BS})$ to $(\mathcal{H}(Z),\rho)$ is given by $(x,t) \mapsto (x,\log D - \log t).$ We use this observation implicitly when making use of the results in \cite{BS}. \subsection{Elementary structure of the hyperbolic cone}
For every $0 \leq r < \infty$, single out the following subsets of $\mathcal{H}(Z)$:
\begin{align}
B_r = Z \times [0,r), \, \overline{B}_r = Z \times [0,r], \textrm{ and } S_r = Z \times \lbrace r \rbrace. \nonumber
\end{align} 
\begin{lemma} \label{help1} Let $(Z,d)$ be a bounded space containing at least two points. Then 
\begin{enumerate}
\item[(i)] the hyperbolic cone $(\mathcal{H}(Z), \rho)$ is $2\mu$-roughly geodesic for some $\mu \geq 0$,
\item[(ii)] for every $x \in Z$ the map $\sigma_x \colon [0, \infty) \rightarrow \mathcal{H}(Z)$ given by $\sigma_x (r) \mapsto (x,r)$ is a geodesic ray in $(\mathcal{H}(Z), \rho)$,
\item[(iii)] if $(Z,d)$ is uniformly coarsely connected then $\mathcal{H}(Z) \setminus B_r$ is uniformly coarsely connected.
\end{enumerate}
\end{lemma}
\begin{proof}
(i) The claim follows by Lemma \ref{vais} observing that for every $x,y \in \mathcal{H}(Z)$ there exists a $\mu$-rough geodesic $\gamma \colon [a,b] \rightarrow \mathcal{H}(Z)$ $x \curvearrowright y$ by \cite[Theorem 7.2]{BS}. \par 
(ii) Fix $x \in Z$ and let $0 \leq r \leq s$. The claim follows from observing that now $$\rho(\sigma_x(r), \sigma_x(s)) = 2 \log \left( \dfrac{e^{-s}}{e^{-(s+r)/2}}\right) = s-r.$$ \par 
(iii) By (ii) we can assume that $t=s=r$. As $(Z,d)$ is $(D(e^{\varepsilon/2}-1)/e^r)$-coarsely connected for every $\varepsilon>0$ the space $(S_r,\rho \vert_{S_r})$ is $\varepsilon$-coarsely connected for every $\varepsilon > 0$ from which the claim follows. 
\end{proof} 
Let $t \geq 0$ and define the projections
\begin{align}
\pi_t \colon \mathcal{H}(Z) \rightarrow S_t \textrm{ by } \pi_t(p,s) = (p,t), \textrm{ and }  h \colon \mathcal{H}(Z) \rightarrow [0, \infty) \textrm{ by } h(p,s) = s. \nonumber 
\end{align}
\begin{lemma} \label{contraction}
$\pi_t \colon \mathcal{H}(Z) \rightarrow S_t$ restricted to $\mathcal{H}(Z) \setminus B_t$ is $1$-Lipschitz.
\end{lemma}
\begin{proof}
Let $(p,r),(q,s) \in \mathcal{H}(Z) \setminus B_t$ and $t \leq s \leq r$. The claim follows observing that
\begin{align}
\rho(\pi_t(p,r),\pi_t(q,s)) &= 2 \log \left( \dfrac{d(p,q)}{e^{-t}D} + 1\right) \leq 2 \log \left( \dfrac{d(p,q)}{e^{-s}D} + 1\right) \nonumber \\ &= 2 \log \left( \dfrac{d(p,q) + \max \lbrace e^{-s},e^{-r} \rbrace D}{e^{-s}D} \right) \nonumber \\ &\leq 2 \log \left( \dfrac{d(p,q) + \max \lbrace e^{-s},e^{-r} \rbrace D}{e^{-(s+r)/2}D} \right) = \rho((p,r),(q,s)). \nonumber 
\end{align}
\end{proof}

\begin{lemma} \label{BB}
If $(p,r) \in \mathcal{H}(Z)$ and $\delta > 0$ then $B((p,r),\delta) \subseteq Z \times (r-\delta, r+\delta).$ In particular if $x,y \in B((p,r),\delta)$ then $\vert h(x)-h(y)\vert < 2 \delta$.
\end{lemma}
\begin{proof}
Let $(q,s) \in B((p,r), \delta)$. The claim follows observing that 
\begin{align}
&\vert r-s\vert = \rho((p,r),(p,s)) = 2 \log \left( \dfrac{\max \lbrace e^{-r},e^{-s} \rbrace D}{e^{-(r+s)/2}D}\right) \nonumber \\ &\leq 2 \log \left( \dfrac{d(p,q) + \max \lbrace e^{-r},e^{-s} \rbrace D}{e^{-(r+s)/2}D}\right) = \rho((p,r), (q,s)) < \delta. \nonumber 
\end{align}
\end{proof}
\subsection{Intrinsic structure of the hyperbolic cone}
By Lemma \ref{help1} the hyperbolic cone $(\mathcal{H}(Z), \rho)$ is $2\mu$-roughly geodesic for some $\mu \geq 0$ and we fix $L(\mu) = 1+2\mu \geq 1$. Define  $\rho_r \colon \mathcal{H}(Z) \setminus B_r \times \mathcal{H}(Z) \setminus B_r \rightarrow [0, \infty]$ for all $r \geq 0$ by $$\rho_{r}(x,y) = \inf \left\lbrace \sum_{i=0}^{n-1} \rho(y_i,y_{i+1}) \colon x=y_0, \dots, y_n=y \, L(\mu)\textrm{-sequence in } \mathcal{H}(Z) \setminus B_r\right\rbrace.$$ 
This replaces $d_r$ in \cite[Section 3]{Cao}. An $L(\mu)$-sequence $x \curvearrowright y$ in $\mathcal{H}(Z) \setminus B_r$ is called an \emph{admissible sequence} for $\rho_{r}(x,y)$.  
\begin{lemma}
If $(Z,d)$ is a bounded uniformly coarsely connected space containing at least two points then $\rho_r$ is a metric on $\mathcal{H}(Z) \setminus B_r$.
\end{lemma}
\begin{proof}
By Lemma \ref{help1} there exists an admissible sequence $x \curvearrowright y$ for any $x,y \in \mathcal{H}(Z) \setminus B_r$ so $\rho_r(x,y) < \infty$. That $\rho_r(x,y) = 0$ if and only if $x=y$ holds as $\rho_r(x,y) = \rho(x,y)$ if $\rho(x,y) \leq L(\mu)$. The rest is clear. 
\end{proof}
The following is left as an elementary exercise in analysis. 
\begin{lemma} \label{inequality}
For any $\varepsilon > 0$ there exists a constant $\kappa(\varepsilon) > 1$ such that $$1+e^{-s}t \leq (1+t)^{\kappa(\varepsilon)^{-s}}$$ for all $s \geq 0$ and all $t \in [0, e^{\varepsilon}]$. \qed
\end{lemma} 
The following now generalises \cite[Lemma 3.1]{Cao}. \begin{proposition} \label{hypinequality}
Suppose $(Z,d)$ is a bounded uniformly coarsely connected space containing at least two points $x,y \in Z$ and $\sigma_x \colon [0,\infty) \rightarrow \mathcal{H}(Z)$, $\sigma_x(t) = (x,t)$, and $\sigma_y \colon [0,\infty) \rightarrow \mathcal{H}(Z)$, $\sigma_y(t) = (y,t)$. Then $$\rho_{r+t}(\sigma_x(r+t), \sigma_y(r+t)) \geq \kappa(L(\mu))^t\rho_r(\sigma_x(r),\sigma_y(r)),$$ for all $r \geq 0$ and all $t \geq 0$.
\end{proposition}
\begin{proof}
Without loss of generality suppose $t > 0$ and let $((p_i,t_i))_i$ be an admissible sequence for $\rho_{r+t}(\sigma_x(r+t), \sigma_y(r+t)).$ Since $\pi_r$ is $1$-Lipshitz by Lemma \ref{contraction}, the sequence $((p_i,r))_i$ is an admissible sequence for $\rho_{r}(\sigma_x(r), \sigma_y(r))$ and
\begin{align}
&\rho((p_i,r),(p_{i+1},r)) = 2 \log \left( \dfrac{d(p_i,p_{i+1}) + e^{-r} D}{e^{-r}D} \right) = 2 \log \left( 1+ e^{-t}\dfrac{d(p_i,p_{i+1})}{e^{-(r+t)}D}\right) \nonumber \\ &\leq 2 \log  \left( 1 + \dfrac{d(p_i,p_{i+1})}{e^{-(r+t)}D} \right)^{\kappa(L(\mu))^{-t}} \leq \kappa(L(\mu))^{-1}\rho((p_i,t_i),(p_{i+1},t_{i+1})) \nonumber
\end{align}
by Lemma \ref{inequality} since $\dfrac{d(p_i,p_{i+1})}{e^{-(r+t)}D} \leq e^{L(\mu)}$ observing that 
\begin{align} \log \left( 1+ \dfrac{d(p_i,p_{i+1})}{e^{-(r+t)}D}\right) &\leq  \log \left( 1+ \dfrac{d(p_i,p_{i+1})}{e^{-(t_i+t_{i+1})/2}D} \right) \nonumber \\ &\leq  \log \left( \dfrac{ \max \lbrace e^{-t_i},e^{-t_{i+1}} \rbrace D + d(p_i,p_{i+1})}{e^{-(t_i + t_{i+1})/2}D} \right) \nonumber \\ &\leq \rho((p_i,t_i)(p_{i+1},t_{i+1})) \leq L(\mu). \nonumber
\end{align} 

The claim now follows observing that $$\rho_r(\sigma_x(r), \sigma_y(r)) \leq \sum_{i=0}^{n-1} \rho ((p_i,r), (p_{i+1},r)) \leq \kappa(L(\mu))^{-t} \sum_{i=0}^{n-1}\rho ((p_i,t_i), (p_{i+1}, t_{i+1})),$$ and taking the infimum over all admissible sequences for $\rho_{r+t}(\sigma_x(r+t), \sigma_y(r+t))$.
\end{proof}

The following lemma now replaces \cite[Assertion 3.1]{Cao}.

\begin{lemma} \label{segment}
Suppose $(Z,d)$ is a bounded uniformly coarsely connected space containing at least two points. If $y=(p,r) \in \mathcal{H}(Z)$ and $t \geq 2L(\mu)$ then $$B((p,r+t), t/(2L(\mu))) \subseteq A^{t/2}_y \times [r,r+2t]$$ where ${A}^{t/2}_y = \lbrace x \in S_r \colon \rho_r(y,x) < t/2\rbrace$. 
\end{lemma}
\begin{proof}
Towards a contradiction, suppose there exists a point \begin{align}z \in B((p,r+t), t/(2L(\mu))) \setminus A^{t/2}_y \times [r,r+2t] \label{tor0}.\end{align} Since $2 L(\mu) \geq 2$, by Lemma \ref{BB} \begin{align} r + t/2 \leq r+t-{t}/{(2L(\mu))} \leq h(z) \leq r+t+{t}/{(2L(\mu))} \leq r + {3t}/{2} \nonumber \end{align} for all $t \geq 2L(\mu)$. As $z \notin A^{t/2}_y \times [r,r+2t]$ \begin{align}\rho_r(y, \pi_r(z)) \geq t/2 \label{tor00},\end{align} for otherwise $\pi_r(z) \in A^{t/2}_y$ and $h(z) < r+2t$ which implies that $z \in A^{t/2}_y \times [r,r+2t]$ after all, contradicting (\ref{tor0}). By Proposition \ref{hypinequality} we now have \begin{align}\rho_{r+t/2}((p,r+t/2),\pi_{r+t/2}(z)) \geq \kappa(L(\mu))^{t/2} \rho_r(y,\pi_r(z)) \geq \kappa(L(\mu))^{t/2} t/2 \label{tor} \end{align} for all $t \geq 2L(\mu)$. Estimating the left-hand side from above we arrive at a contradiction completing the proof. Towards this,
\begin{align}
&\rho_{r+t/2}((p,r+t/2), \pi_{r+t/2}(z)) \nonumber \\ \leq & \rho_{r+t/2}((p,r+t/2), (p,r+t)) + \rho_{r+t/2}((p,r+t), \pi_{r+t/2}(z)) \nonumber \\ \leq & t/2 + \rho_{r+t/2}((p,r+t),z) + \rho_{r+t/2}(z,\pi_{r+t/2}(z)) \nonumber \\ \leq & t/2 + \rho_{r+t/2}((p,r+t),z) + 3t/2-t/2 \nonumber \\ =& 3t/2 + \rho_{r+t/2}((p,r+t),z) \label{tor000}. \end{align}
To estimate $\rho_{r+t/2}((p,r+t),z)$ from above, let $\gamma \colon [0,\rho((p,r+t),z)] \rightarrow \mathcal{H}(Z)$ be a $2\mu$-rough geodesic $(p,r+t) \curvearrowright z$ by Lemma \ref{help1}, fix $m \in \mathbb{N}$ such that $m-1 \leq \rho((p,r+t),z) \leq m$, and let $x_k = \gamma((k\rho((p,r+t),z)/m))$ for $k \in \lbrace 0, \dots, m\rbrace \subseteq \mathbb{N}$. We claim that $(x_k)_k$ is an admissible sequence for $\rho_{r+t/2}((p,r+t),z)$. To begin, $(x_k)_k$ is an $L(\mu)$-sequence $(p,r+t) \curvearrowright z$ of length $m$ since 
 
$$\rho(x_k, x_{k+1}) \leq \rho((p,r+t),z)/m + 2\mu \leq 1 + 2\mu = L(\mu).$$ The sequence is admissible if $x_k \in \mathcal{H}(Z) \setminus B_{r+t/2}$ for all $k \in \lbrace 0, \dots ,m\rbrace$. To see that this is the case, note that if $h(x_k) < r + t/2$ then $\rho(x_0,x_k) > t/2$ and $$t/2 < \rho(x_0,x_k) \leq k\rho((p,r+t),z)/m + 2\mu \leq \rho((p,r+t),z) + 2\mu \leq t/(2L(\mu)) + 2\mu$$ for all $t \geq 2L(\mu)$ which is not possible. Thus,  
$$\rho_{r+t/2}((p,r+t),z) \leq m L(\mu) \leq (t/(2L(\mu)) +1)L(\mu),$$ which together with (\ref{tor000}) gives that $$\rho_{r+t/2}((p,r+t/2), \pi_{r+t/2}(z)) \leq 3t/2 + (t/(2L(\mu))+1)L(\mu) \leq 5 L(\mu) t / 2$$ for all $t \geq 2L(\mu)$. Together with (\ref{tor}) this implies that $5L(\mu) \geq \kappa(L(\mu))^{t /2}$ for all $t \geq 2L(\mu)$ which is impossible. Thus, $z$ as in (\ref{tor0}) can not exist and the claim follows.
\end{proof}

\subsection{Cao's graph structure} \label{GS}
In this section we approximate the hyperbolic cone by a graph structure due to Cao in \cite{Cao}. Here by a \emph{graph} we mean a $1$-dimensional abstract simplicial complex $\Gamma$ whose $0$-simplexes are its \emph{vertices} and its $1$-simplexes its \emph{edges}. We write $\Gamma^{(0)}$ for the \emph{set of vertices} and $\Gamma^{(1)}$ for the \emph{set of edges}, and whenever $\lbrace u,v\rbrace \in \Gamma^{(1)}$ we say that $u$ and $v$ are \emph{neighbours} and write $u \sim v$. Let $N(v) = \lbrace u \colon u \sim v \rbrace$. If for some constant $c \in \mathbb{N}$ it holds that $\#N(v) \leq c$ for all $v \in\Gamma^{(0)}$ we say that $\Gamma$ has \emph{bounded valency (by $c$)}. \par 
A \emph{graph structure $(\Gamma X,d_{\Gamma})$ on $(X,d)$} is a pair where $\Gamma X$ is a graph with vertex set $\Gamma X^{(0)} = X$ and $d_\Gamma \colon X \times X \rightarrow [0,\infty]$ is given by
\begin{enumerate}
\item $d_\Gamma(x,y) = 0$ if and only if $x=y$,
\item $d_\Gamma(x,y) = n$ if the shortest edge path $x \curvearrowright y$ is of length $n$,
\item $d_\Gamma(x,y) = \infty$ if there is no edge path $x \curvearrowright y$, 
\end{enumerate} 
where an edge path $x \curvearrowright y$ (of length $n \in \mathbb{N}$) is any finite sequence $x=x_0, \dots, x_n=y$ of points in $X$ such that $x_{i} \sim x_{i+1}$ for all $0 \leq i \leq n-1$.
 
\subsubsection*{Cao's graph structure} \label{caocao}
Suppose $(\mathcal{H}(Z),\rho)$ is $2\mu$-roughly geodesic and uniformly coarsely proper for $R > r > r_b$ and fix $\delta > 0$ and $r_0 >0 $ such that  
\begin{align} r_0/3 >& \delta >  c(\mu)(r_b+1) \tag{$\theta_0$} \textrm{ and }\label{RR0} \\ \kappa(L(\mu))^{r_0} >&  8 \delta N(10\delta, \delta/c(\mu)) \tag{$\theta_1$} \label{RRR0} \end{align}
 where $c(\mu) = 2L(\mu) \geq 2$ hold.  
For $i \in \mathbb{N}$, let $N_{ir_0} = \lbrace (p_{i,\alpha},ir_0) \colon \alpha \in \mathcal{I}_{i} \rbrace$ be a maximal $\delta$-net in $(S_{ir_0}, \rho_{ir_0})$ indexed by $\mathcal{I}_i$ and write 
\begin{align}
q_{i,\alpha} &= (p_{i, \alpha},ir_0), \,
v_{i,\alpha} = \pi_{ir_0+\delta}(q_{i,\alpha}), \,
A(q_{i,\alpha}) = B_{\rho_{ir_0}}(q_{i,\alpha},3 \delta) \cap S_{ir_0}, \nonumber \\
V(v_{i,\alpha}) &=A(q_{i,\alpha}) \times [ir_0, (i+1)r_0]. \nonumber
\end{align}
The graph structure $(\Gamma \mathcal{H}(Z), d_\Gamma)$ where \begin{align}\Gamma\mathcal{H}(Z)^{(0)} = \bigcup_{i \in \mathbb{N}} \pi_{ir_0 + \delta}\left( N_{ir_0}\right) \, \textrm{and } \Gamma\mathcal{H}(Z)^{(1)} = \lbrace \lbrace u, v \rbrace \colon V(u) \cap V(v) \neq \emptyset \rbrace \nonumber\end{align} is called \emph{Cao's graph structure} and $\Gamma \mathcal{H}(Z)$ the \emph{Cao graph}. \subsection{Basic properties of Cao's graph structure}
We now prove that Cao's graph structure approximates the hyperbolic cone. 

\begin{proposition} \label{BIG2}
Let $(Z,d)$ be a bounded uniformly coarsely connected space containing at least two points with uniformly coarsely proper hyperbolic cone $(\mathcal{H}(Z),\rho)$. Then
\begin{enumerate}
\item[(i)] $\Gamma \mathcal{H}(Z)^{(0)}$ is $\delta/c(\mu)$-separated in $(\mathcal{H}(Z), \rho)$, 
\item[(ii)] $\Gamma \mathcal{H}(Z)^{(0)}$ is $2r_0$-cobounded in $(\mathcal{H}(Z), \rho)$,
\item[(iii)] $(\Gamma \mathcal{H}(Z), d_\Gamma)$ is quasi-isometric to $(\mathcal{H}(Z),\rho)$,
\item[(iv)] $\Gamma\mathcal{H}(Z)^{(0)}$ is countable and $\Gamma\mathcal{H}(Z)$ has bounded valency by $N(10r_0,\delta/c(\mu))$.
\end{enumerate}
\end{proposition}

\begin{proof}
(i) Suppose $v \in \Gamma \mathcal{H}(Z)^{(0)}$ where $v= \pi_{ir_0+ \delta}(q)$ for $q \in N_{ir_0}$. By Lemma \ref{segment}
\begin{align} B(v,\delta/c(\mu)) &\subseteq A^{\delta/2}_{q} \times [h(q),h(q)+2 \delta], \nonumber  
\end{align} and $$(A^{\delta/2}_{q} \times [h(q), h(q) + 2 \delta]) \cap (A^{\delta/2}_{p} \times [h(p),h(p)+2 \delta]) = \emptyset$$ if $q \in N_{ir_0}$ and $p \in N_{jr_0}$ are distinct points since $N_{ir_0}$ is a maximal $\delta$-net in $(S_{ir_0},\rho_{ir_0})$ and $r_0 > 3 \delta$ by (\ref{RR0}). Hence $\rho(u,v) \geq \delta/c(\mu)$ if $u$ and $v$ are distinct vertices in the Cao graph. The claim now follows. \par (ii) Let $z \in \mathcal{H}(Z)$, $i \in \mathbb{N}$ such that $ir_0 \leq h(z) < (i+1)r_0$, and $q \in N_{ir_0}$ such that $\rho_{ir_0}(\pi_{ir_0}(z), q) \leq \delta$. Now
\begin{align}
&\rho(z, \Gamma \mathcal{H}(Z)^{(0)}) \leq \rho(z, \pi_{ir_0+\delta}(q)) \leq \rho(z,\pi_{ir_0}(z)) + \rho(\pi_{ir_0}(z), \pi_{ir_0+\delta}(q)) \nonumber \\ &\leq r_0 + \rho_{i_0r_0}(\pi_{ir_0}(z), q) +\rho(q,\pi_{ir_0+\delta}(q))  \leq r_0 + \delta + \delta  < 2r_0 \nonumber
\end{align}
since $r_0 > 3 \delta$ by (\ref{RR0}). The claim now follows. \par (iii) By (ii) it suffices to show that the inclusion $(\Gamma \mathcal{H}(Z)^{(0)}, d_\Gamma) \hookrightarrow (\mathcal{H}(Z), \rho)$ is a quasi-isometric embedding. Explicitly, we prove that \begin{align}\dfrac{1}{8r_0} \rho(u, v) \leq d_\Gamma(u, v) < 3r_0 \rho(u, v) \label{qi0} \end{align} for all $u,v \in \Gamma\mathcal{H}(Z)^{(0)}$. We begin by proving the right-hand side of (\ref{qi0}). Let $u,v \in \Gamma\mathcal{H}(Z)^{(0)}$ be distinct vertices, $\gamma \colon [0,r] \rightarrow \mathcal{H}(Z)$ a $2\mu$-rough geodesic $u \curvearrowright v$ where $r=\rho(u,v)$ which exists by Lemma \ref{help1}, and $m \in \mathbb{N}$ such that $m-1 < r \leq m$. Now, $$\rho(\gamma(kr/m),\gamma((k+1)r/m)) \leq r/m + 2\mu \leq 1 + 2\mu = L(\mu)$$ for every $k \in \lbrace 0, \dots, m-1 \rbrace \subseteq \mathbb{N}$. For each $k \in \lbrace 0, \dots, m\rbrace \subseteq \mathbb{N}$ choose $q_{i(k),\alpha(k)} \in N_{i(k)r_0}$ such that $\pi_{i(0)+\delta}(q_{i(0),\alpha(0)})=u$, $\pi_{i(m)+\delta}(q_{i(m),\alpha(m)})=v$, and \begin{align}i(k)r_0 \leq h(\gamma(kr/m)) &< (i(k) + 1)r_0,\nonumber \\ \rho_{i(k)r_0}(q_{i(k),\alpha(k)}, \pi_{i(k)r_0}\gamma(kr/m)) &< \delta \nonumber, \nonumber \end{align} and write $v_{i(k),\alpha(k)}=\pi_{i(k)r_0 + \delta}(q_{i(k),\alpha(k)})$ as usual. Let $i_0 = \min \lbrace i(k), i(k+1) \rbrace$. 
Since the restriction of $\pi_{i_0r_0}$ to $\mathcal{H}(Z) \setminus B_{i_0r_0}$ is $1$-Lipschitz by Lemma \ref{contraction}, $$\rho(\pi_{i_0r_0}\gamma(i(k)r/m), \pi_{i_0r_0}\gamma(i(k+1)r/m)) \leq L(\mu)$$ so $\rho_{i_0r_0}(\pi_{i_0r_0}\gamma(i(k)r/m), \pi_{i_0r_0}\gamma(i(k+1)r/m)) \leq L(\mu).$ Choose $p, q \in N_{i_0r_0}$ such that
\begin{align}
\rho_{i_0r_0}(p, \pi_{i_0r_0}\gamma(i(k)r/m)) &< \delta, \nonumber \\ \rho_{i_0r_0}(q, \pi_{i_0r_0}\gamma(i(k+1)r/m)) &< \delta, \nonumber
\end{align}
and note that by Lemma \ref{contraction} and (\ref{RR0})
\begin{align}
\rho_{ir_0}(p, q) &\leq  \rho_{i_0r_0}(p, \pi_{i_0r_0}\gamma(i(k)r/m)) \nonumber \\ &+ \rho_{i_0r_0}(\pi_{i_0r_0}\gamma(i(k)r/m), \pi_{i_0r_0}\gamma(i(k+1)r/m)) \nonumber \\ &+ \rho_{i_0r_0}(\pi_{i_0r_0}\gamma(i(k+1)r/m), q) \nonumber \\ &< \delta + L(\mu) + \delta < 3 \delta. \nonumber
\end{align} 
Thus $p \in A(q)$ so $\pi_{i_0r_0+\delta}(p) \in V(\pi_{i_0r_0 + \delta}(q))$ and $d_\Gamma(\pi_{i_0r_0+\delta}(p), \pi_{i_0r_0+\delta}(q)) \leq 1$ giving
\begin{align} d_\Gamma(v_{i(k),\alpha(k)}, &v_{i(k+1),\alpha(k+1)}) \leq \nonumber \\  &1 + d_\Gamma(v_{i(k),\alpha(k)}, \pi_{i_0r_0 + \delta}(p)) + d_\Gamma(\pi_{i_0r_0 + \delta}(q), v_{i(k+1),\alpha(k+1)}). \nonumber \end{align}
However, since $\vert i(k)-i(k+1)\vert \leq 1$ \begin{align}\pi_{i(k)r_0}(\gamma(kr/m)) &\in V(v_{i(k), \alpha(k)}) \cap V(\pi_{i_0r_0+\delta}(p)), \nonumber \\ \pi_{i(k+1)r_0}(\gamma(i(k+1)r/m)) &\in V(v_{i(k+1), \alpha(k+1)}) \cap V(\pi_{i_0r_0+\delta}(q)), \nonumber\end{align} so $d_\Gamma(v_{i(k),\alpha(k)}, \pi_{i_0r_0 + \delta}(p)) \leq 1$ and $d_\Gamma(v_{i(k+1),\alpha(k+1)}, \pi_{i_0r_0 + \delta}(q)) \leq 1$, and altogether $$d_\Gamma(v_{i(k),\alpha(k)}, v_{i(k+1),\alpha(k+1)}) \leq 3.$$ Finally \begin{align} d_\Gamma(u,v) &\leq \sum_{k=0}^{m-1}d_\Gamma(v_{i(k), \alpha(k)}, v_{i(k+1), \alpha(k+1)}) \leq 3m \leq 3(r+1) \nonumber \\ &= 3\rho(u,v) +3 \leq 3 (1+ c(\mu)/\delta) \rho(u,v)  < 3r_0 \rho(u,v)  \nonumber \end{align} as $\rho(u,v) \geq \delta/c(\mu)$ by (i) which gives the the right-hand side of (\ref{qi0}). To prove the left-hand side of (\ref{qi0}) let $u,v \in \Gamma \mathcal{H}(Z)^{(0)}$ be two vertices. Without loss of generality, assume that $d_\Gamma(u,v) = n \in \mathbb{N} \setminus \lbrace 0\rbrace$ realised by the edge path $u=x_0, \dots, x_n=v$. Since $x_i \sim x_{i+1}$ it follows that $V(x_i) \cap V(x_{i+1}) \neq \emptyset$ where $\mathrm{diam}(V(x_i)) < 4r_0$ for all $0 \leq i \leq n-1$. Thus $\rho(x_i,x_{i+1}) < 8r_0$ for all $0 \leq i \leq n-1$ and    
\begin{align}
\rho(u,v) \leq \sum_{k=0}^{n-1} \rho(x_i, x_{i+1}) < 8r_0n = 8r_0d_\Gamma(u,v), \nonumber 
\end{align}
which gives the left-hand side of (\ref{qi0}) and the claim follows. \\ \par (iv) For $n \in \mathbb{N}$ let \begin{align} C_n = \left\lbrace B(v, \delta/c(\mu)) \colon v \in \Gamma \mathcal{H}(Z)^{(0)} \textrm{ and } h(v) \leq nr_0 + \delta \right\rbrace \textrm{ so } \Gamma \mathcal{H}(Z)^{(0)} \subseteq \bigcup_{n \in \mathbb{N}} \bigcup C_n. \nonumber \end{align} We claim that $\#C_n < \infty$ for every $n \in \mathbb{N}$ from which the claim then follows. By Lemma \ref{BB} for any $z \in S_0$ and $n \in \mathbb{N}$ $$\bigcup C_n \subseteq B\left(z, nr_0 + {\delta}/{c(\mu)} + \delta + \mathrm{diam}\left(S_{nr_0 + \delta/c(\mu) + \delta} \right) \right) \subseteq B(z, 5(n+2)r_0)$$ using (\ref{RR0}) together with $$\mathrm{diam}(S_{nr_0 + \delta/c(\mu) + \delta}) \leq 2 \log \left( \dfrac{D + e^{-(nr_0+\delta/c(\mu) + \delta)}D}{e^{-(nr_0 + \delta/c(\mu)+\delta)}D}\right) = 2 \log(e^{nr_0 + \delta/c(\mu) + \delta} + 1).$$ Let $R(n) = 5(n+2)r_0$. Since $R(n) > r_b$ by (\ref{RR0}) and $(\mathcal{H}(Z),\rho)$ is uniformly coarsely proper $B(z,R(n))$ is covered by $N(R(n),\delta/c(\mu))$ balls of radius $\delta/c(\mu)$ and $\# C_n \leq N(R(n),\delta/c(\mu))$ by part (i) and it follows that $\Gamma \mathcal{H}(Z)^{(0)}$ is countable. To see that $\Gamma \mathcal{H}(Z)$ has bounded valency note that if $v \sim u$ then $d_\Gamma(v, u) \leq 1$ and by inequality $(\ref{qi0})$ above $\rho(v, u) \leq 8r_0.$ In particular $$B(u, \delta/c(\mu)) \subseteq B(v, 9r_0 + \delta/c(\mu)),$$ and $B(v, \delta/c(\mu)) \cap B(v, \delta/c(\mu)) = \emptyset$ by part (i) if $v$ and $u$ are distinct vertices in the Cao graph. Once again, since $(\mathcal{H}(Z),\rho)$ is uniformly coarse proper $$\# N(v) \leq N(9r_0+\delta/c(\mu),\delta/c(\mu)) \leq N(10r_0,\delta/c(\mu))$$ from which the claim follows.
\end{proof}
 
The following lemma now replaces \cite[Assertion 3.2]{Cao}.
\begin{lemma} \label{technical2}
Let $(Z,d)$ be a bounded uniformly coarsely connected space containing at least two points with uniformly coarsely proper hyperbolic cone $(\mathcal{H}(Z),\rho)$. Then for any $i \in \mathbb{N}$ and any $q \in S_{ir_0}$ $$\# V(i,q) \leq N(10\delta,\delta/c(\mu))$$ where $V(i,q) = \lbrace v_{i,\alpha} \in \Gamma\mathcal{H}(Z)^{(0)} \colon \rho_{ir_0}(q_{i, \alpha}, q) < 4 \delta \rbrace$.     
\end{lemma}
\begin{proof}
Suppose $v_{i,\beta} \in V(i,q)$. As $A^{\delta/2}_{q_{i,\beta}} \subseteq A(q_{i,\beta})$ and $\rho \leq \rho_{ir_0}$ $$B(v_{i,\beta}, \delta/c(\mu)) \subseteq A(q_{i,\beta}) \times [ir_0, ir_0 + 2 \delta] \subseteq B(q,10\delta)$$ by Lemma \ref{segment}. By Proposition \ref{BIG2} the balls $B(v_{i,\alpha}, \delta/c(\mu))$ and $B(v_{i,\beta}, \delta/c(\mu))$ are disjoint if $v_{i,\alpha} \neq v_{i,\beta}$. Thus $\# V(i,q) \leq N(10\delta, \delta/c(\mu))$ since $(\mathcal{H}(Z),\rho)$ is uniformly coarsely proper and $\delta/c(\mu) > r_b$.  
\end{proof}
We use this to find a uniform upper bound for the downward flow in the Cao graph.
\begin{proposition} \label{asymmetry1}
Let $(Z,d)$ be a bounded uniformly coarsely connected space containing at least two points with uniformly coarsely proper hyperbolic cone $(\mathcal{H}(Z),\rho)$. Let $v \in \Gamma\mathcal{H}(Z)^{(0)}$ and $ N^-(v) = \lbrace w \in \Gamma \mathcal{H}(Z)^{(0)} \colon w \sim v \textrm{ and } h(w) = h(v) - r_0\rbrace$. Then $$\#N^-(v) \leq N(10\delta, \delta/c(\mu))$$ for all $v \in \Gamma \mathcal{H}(Z)^{(0)}$.
\end{proposition}
\begin{proof}
Fix $v_{i,\alpha} \in \Gamma \mathcal{H}(Z)^{(0)}$. If $i=0$ then $N^-(v_{i,\alpha}) = \emptyset$ so assume $i \geq 1$ and let $v_{j,\beta} \in N^-(v_{i,\alpha})$. Then $j = i - 1$ and $V(v_{i,\alpha}) \cap V(v_{i-1,\beta}) \neq \emptyset$. In particular, there exists $y \in A(q_{i,\alpha})$ such that $\rho_{ir_0}(q_{i,\alpha},y) < 3 \delta$ and
\begin{align}\rho_{(i-1)r_0}(\pi_{(i-1)r_0}(y),\pi_{(i-1)r_0}(q_{i,\alpha})) &\leq \kappa(L(\mu))^{-r_0} \rho_{ir_0}(y,q_{i,\alpha}) < 3 \delta \kappa(L(\mu))^{-r_0} \nonumber \\ &< \delta \nonumber \end{align} 
by Proposition \ref{hypinequality} and (\ref{RRR0}). As $y \in V(v_{i-1,\beta})$
\begin{align}
\rho_{(i-1)r_0}(q_{i-1,\beta},\pi_{(i-1)r_0}(q_{i,\alpha})) &\leq  \rho_{(i-1)r_0}(q_{i-1,\beta},\pi_{(i-1)r_0}(y)) \nonumber \\ &+ \rho_{(i-1)r_0}(\pi_{(i-1)r_0}(y),\pi_{(i-1)r_0}(q_{i,\alpha})) \nonumber \\ &< 3 \delta + \delta = 4 \delta \nonumber
\end{align}
so $q_{i-1,\beta} \in V(i-1, \pi_{(i-1)r_0}(q_{i,\alpha}))$ and so $\#N^-(v_{i,\alpha}) \leq N(10\delta,\delta/c(\mu))$ by Lemma \ref{technical2}.
\end{proof}
The following gives a uniform lower bound for the upward flow in the Cao graph.
\begin{proposition} \label{asymmetry2}
Let $(Z,d)$ be a bounded uniformly coarsely connected space containing at least two points with uniformly coarsely proper hyperbolic cone $(\mathcal{H}(Z),\rho)$. Let $v \in \Gamma\mathcal{H}(Z)^{(0)}$ and $N^+(v) = \lbrace w \in \Gamma \mathcal{H}(Z)^{(0)} \colon w \sim v \textrm{ and } h(w) = h(v) + r_0\rbrace$. Then $$2N(10\delta,\delta/c(\mu)) \leq \#N^+(v)$$ for all $v \in \Gamma \mathcal{H}(Z)^{(0)}$ with $h(v) > \delta$.
\end{proposition}
\begin{proof}
Let $v_{i,\alpha} \in \Gamma \mathcal{H}(Z)^{(0)}$ such that $\delta < h(v_{i,\alpha}) = ir_0+ \delta$. Now $i \geq 1$ and \begin{align} \mathrm{diam}_{\rho_{ir_0}}(S_{ir_0}) &\geq \mathrm{diam}_\rho (S_{ir_0}) \geq 2 \log \left( \dfrac{D + e^{-ir_0}D}{e^{-ir_0}D}\right) = 2 \log \left( e^{ir_0} + 1\right) \geq 2ir_0  \nonumber \\ &\geq 2r_0 > 6\delta \nonumber\end{align}
by (\ref{RR0}). Fix $m = 2N(10 \delta,\delta/c(\mu))$ and let $k \in \lbrace 0, \dots, m\rbrace \subseteq \mathbb{N}$. Now, for all $0 \leq k/m \leq 1$ there exist $x_{k/m} \in S_{ir_0}$ such that $$k/m-\varepsilon \leq \rho(q_{i,\alpha}, x_{k/m}) \leq k/m + \varepsilon$$ for any $0<\varepsilon < 1/(4m)$ since $(S_{ir_0}, \rho \vert_{S_{ir_0}})$ is uniformly coarsely connected. In particular if ${k_1/m} \neq k_2/ m$, say $k_1 > k_2$, then 
\begin{align}
&\rho_{ir_0}(x_{k_1/m},x_{k_2/m}) \geq \rho(x_{k_1/m},x_{k_2/m}) \geq \rho(x_{k_1/m},q_{i,\alpha}) - \rho(q_{i,\alpha}, x_{k_2/m}) \nonumber \\ &\geq k_1/m- \varepsilon -(k_2/m + \varepsilon) \geq (k_1 - k_2)/m - 2 \varepsilon \geq 1/m - 2 \varepsilon > 1/(2m) \nonumber 
\end{align}
as $0<\varepsilon< 1/(4m)$. For each $k \in \lbrace 0, \dots m\rbrace$ let $y_k = \pi_{(i+1)r_0}(x_{k/m})$. As previously for ${k_1} > {k_2}$ \begin{align}\rho_{(i+1)r_0}(y_{k_1}, y_{k_2}) &= \rho_{(i+1)r_0} (\pi_{(i+1)r_0}(x_{k_1/m}), \pi_{(i+1)r_0}(x_{k_2/m})) \nonumber \\ &\geq \kappa(L(\mu))^{r_0}\rho_{ir_0}(x_{k_1/m}, x_{k_2/m}) \geq \dfrac{\kappa(L(\mu))^{r_0}}{2m} \nonumber \\ &= \dfrac{\kappa(L(\mu))^{r_0}}{4N(10\delta, \delta/ c(\mu))} > 2 \delta \nonumber \end{align}
by Lemma \ref{hypinequality} and (\ref{RRR0}). Now for each $y_k$ choose $q_{k} \in N_{(i+1)r_0}$ such that $$\rho_{(i+1)r_0}(y_k, q_k) < \delta$$ noting that $\pi_{(i+1)r_0 + \delta}(q_k) \sim v_{i,\alpha}$ since $y_k \in V(\pi_{(i+1)r_0 + \delta}(q_k)) \cap V(v_{i,\alpha})$. Moreover, $\pi_{(i+1)r_0 + \delta}(q_{k_1}) \neq \pi_{(i+1)r_0 + \delta}(q_{k_2})$ whenever ${k_1} > {k_2}$ since
\begin{align}
\rho_{(i+1)r_0}(q_{k_1}, q_{k_2}) &\geq \rho_{(i+1)r_0}(y_{k_1}, y_{k_2})- \rho_{(i+1)r_0}(y_{k_1}, q_{k_1}) \nonumber \\ &- \rho_{(i+1)r_0} (q_{k_2}, y_{k_2}) > 2 \delta - 2 \delta = 0. \nonumber
\end{align}
Thus $\lbrace 0, \dots m\rbrace \rightarrow N^+(v_{i,\alpha})$ for $k \mapsto \pi_{(i+1)r_0 + \delta}(q_k)$ is an injection and $\# N^+(v_{i,\alpha}) \geq 2N(10\delta, \delta/c(\mu))$ for $i \geq 1$ proving the claim.
\end{proof}

\subsection{Non-amenability of the hyperbolic cone}
Let $(\mathcal{H}(Z),\rho)$ be uniformly coarsely proper, let $\mathbb{R}^{\Gamma \mathcal{H}(Z)^{(0)}} = \lbrace f \colon \Gamma\mathcal{H}(Z)^{(0)} \rightarrow \mathbb{R} \rbrace$, and let $\Delta \colon \mathbb{R}^{\Gamma \mathcal{H}(Z)^{(0)}} \rightarrow \mathbb{R}^{\Gamma \mathcal{H}(Z)^{(0)}}$ be the \emph{graph Laplacian} given by $$\Delta f(v) = \dfrac{1}{\# N(v)} \left(\sum_{w \sim v} f(w)\right) - f(v).$$  \begin{lemma} \label{finallemma}
Let $(Z,d)$ be a bounded uniformly coarsely connected space containing at least two points with uniformly coarsely proper hyperbolic cone $(\mathcal{H}(Z),\rho)$. If there exist a Lipschitz function $f \colon \Gamma\mathcal{H}(Z)^{(0)} \rightarrow \mathbb{R}$ and $C> 0$ such that $\Delta f(v) > C $ for every $v \in \Gamma \mathcal{H}(Z)^{(0)}$ then $(\mathcal{H}(Z), \rho)$ is non-amenable. 
\end{lemma}
\begin{proof}
By Proposition \ref{BIG2} the assumptions in \cite[Proposition 2.3]{Cao} hold so the Cheeger constant of $\Gamma \mathcal{H}(Z)$ is strictly positive, equivalently, $(\Gamma \mathcal{H}(Z),d_{\Gamma})$ is non-amenable. The claim follows as $(\mathcal{H}(Z), \rho)$ and $(\Gamma \mathcal{H}(Z),d_{\Gamma})$ are quasi-isometric. 
\end{proof}
\begin{main} \label{BIGTHEOREM1}
Let $(\mathcal{H}(Z),\rho)$ be the hyperbolic cone over a bounded space $(Z,d)$. If $(\mathcal{H}(Z), \rho)$ is uniformly coarsely proper and $(Z,d)$ consists of a finite union of uniformly coarsely connected components each containing at least two points then $(\mathcal{H}(Z), \rho)$ is non-amenable. 
\end{main} 
\begin{proof}
First suppose $(Z,d)$ is uniformly coarsely connected and contains at least two points. Since $\vert h(v_{i,\alpha})-h(v_{j,\beta})\vert = \vert ir_0 + \delta - jr_0 - \delta \vert \leq \vert i-j\vert r_0 \leq r_0 $ whenever $v_{i,\alpha} \sim v_{j,\beta}$ it follows that $h$ is $r_0$-Lipschitz on $\Gamma \mathcal{H}(Z)^{(0)}$. Moreover, if $i \geq 1$ 
\begin{align} 
\dfrac{\Delta h(v_{i,\alpha})}{r_0} &= \dfrac{1}{\# N(v_{i,\alpha})r_0}\sum_{v_{j,\beta} \sim v_{j,\alpha}} (h(v_{j,\beta}) - h(v_{i,\alpha})) = \dfrac{\#N^+(v_{i,\alpha}) - \#N^-(v_{i,\alpha})}{\# N(v_{i,\alpha})} \nonumber \\ &\geq \dfrac{2N(10\delta,\delta/c(\mu)) - N(10\delta,\delta/c(\mu))}{N(10r_0,\delta/c(\mu))} \geq \dfrac{1}{N(10r_0,\delta/c(\mu))} > 0 \nonumber 
\end{align}
by Lemma \ref{asymmetry1} and Lemma \ref{asymmetry2} where $\# N(v_{i,\alpha}) \leq N(10r_0, \delta/c(\mu))$ for all $v_{i,\alpha} \in \Gamma \mathcal{H}(Z)^{(0)}$ by Lemma \ref{BIG2}. If $i=0$ we have $N^-(v_{i,\alpha}) = \emptyset$ and the same lower bound holds for $\Delta h$. Thus $(\mathcal{H}(Z),\rho)$ is non-amenable by Lemma \ref{finallemma}. \par
Now suppose that $(Z,d)$ is a finite union of uniformly coarsely connected components $Z=Z_1 \sqcup \dots \sqcup Z_n$ where each component $Z_i$ contains at least two points. To see that $(\mathcal{H}(Z),\rho)$ is non-amenable let $\Gamma = \Gamma_1 \sqcup \dots \sqcup \Gamma_n \subseteq \mathcal{H}(Z)$ be a quasi-lattice in $(\mathcal{H}(Z),\rho)$ such that $\Gamma_i \subseteq \mathcal{H}(Z_i)$ is a quasi-lattice in $(\mathcal{H}(Z_i), \rho \vert_{\mathcal{H}(Z_i)})$, and let $F \subseteq \Gamma$ be any finite set and write $F_i = F \cap \mathcal{H}(Z_i)$ so that $F = F_1 \sqcup \dots \sqcup F_n$. By the first part of the proof each $(\mathcal{H}(Z_i),\rho \vert_{\mathcal{H}(Z_i)})$ is non-amenable, so for some constants $C_i >0$ and $r_i>0$ the isoperimetric inequality $\# F_i \leq C_i \# \partial_{r_i} F_i$ holds and hence  
\begin{align}
\# F = \#F_1 + \dots + \#F_n \leq C_1 \# \partial_{r_1} F_1 + \dots C_n \# \partial_{r_n} F_n \leq C \# \partial_{r} F \nonumber  
\end{align} for $C = \max \lbrace C_1, \dots, C_n\rbrace$ and $r = \max\lbrace r_1, \dots, r_n\rbrace$. The claim now follows. 
\end{proof}

\end{document}